\documentclass{article}

\usepackage{fullpage}
\usepackage{amsmath,amssymb,amsthm,amscd,mathtools}
\usepackage{enumerate}
\usepackage{tikz}
\usetikzlibrary{intersections,through}

\newcommand{\pgfextractangle}[3]{%
	\pgfmathanglebetweenpoints{\pgfpointanchor{#2}{center}}
		{\pgfpointanchor{#3}{center}}
	\global\let#1\pgfmathresult  
}

\newcommand{\CNK}[3]	{
\begin{tikzpicture}[scale=#3] 
\def\n{#1} 
\def\k{#2}
\edef\m{\number\numexpr(\n-1)/2\relax}
\edef\l{\number\numexpr \k-1 \relax}
\edef\j{\number\numexpr 2*\n-1 \relax}
\draw (0,0) circle (1);
\foreach \z in {0,...,\j}	{\foreach \y in {0,...,\l}	{\draw (\z*180/\n:1) \foreach \x in {1,...,\m}	{arc(\z*180/\n+(2*\x-1)*180/\n-\y*180/\n/\k:\z*180/\n+2*\x*180/\n-\y*180/\n/\k:1)};}}
\end{tikzpicture}
}

\newcommand{\CNKTa}[3]	{	
\begin{tikzpicture}[scale=#3]
\def\k{#1}
\def\t{#2}
\edef\m{\number\numexpr \k-1 \relax}
\draw (0,0) circle (1);
\foreach \x in {1,...,6}	{
\draw (60*\x:1) -- ++(60*\x+180:\t) -- ++(60*\x+60:\t) arc (60*\x+60:60*\x+120:1) -- (0,0);
\foreach \y in {1,...,\m} {\draw (60*\x:1-\t) -- ++(60*\x+\y*60/\k:\t) arc(60*\x+\y*60/\k:60*\x+\y*60/\k+60:1) -- ++(60*\x+240+\y*60/\k:\t) -- ++(60*\x+120+\y*60/\k:\t);}}
\end{tikzpicture}
}

\newcommand{\C}[3]	{
\begin{tikzpicture}[scale=#3]
\draw (0,0) circle (1);
\foreach \x in {#1}	{\draw (30*\x:1) arc (60+30*\x:120+30*\x:1);}
\foreach \x in {#2}	{\draw (30*\x-30:1) arc (30*\x:60+30*\x:1);}
\end{tikzpicture}
}

\newcommand{\N}{\mathbb{N}}
\newcommand{\n}{\mathcal{N}}

\newenvironment{definition}{\medskip\noindent\textbf{Definition:} }{\medskip}

\newtheorem{theorem}{Theorem}[section]
\newtheorem{lemma}[theorem]{Lemma}
\newtheorem{prop}[theorem]{Proposition}
\newtheorem{cor}[theorem]{Corollary}
\newtheorem{conj}[theorem]{Conjecture}

\title{Infinite families of monohedral disk tilings}
\author{Joel Anthony Haddley, Stephen Worsley\\University of Liverpool}
\date{}

\begin{document}

\maketitle

\section{Introduction}

A tiling of a planar shape is called {\em monohedral} if all tiles are congruent to each other. Our study will be monohedral tilings of the disk. Such tilings are produced on a daily basis by pizza chefs taking radial cuts distributed evenly around the centre of the pizza (Figure~\ref{fig:symradgen}). We will call such tilings {\em symmetric radially generated tilings}, such

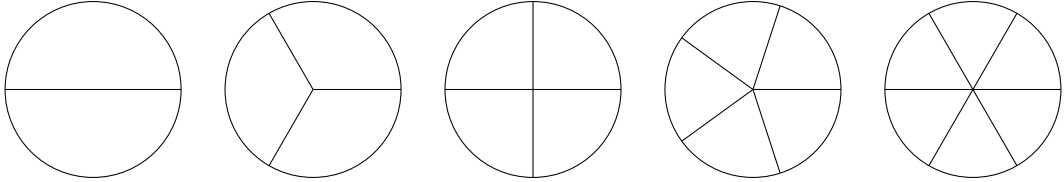
\begin{figure}[htbp]\centering
\begin{tikzpicture}[scale=1.3]
\def\r{0.9}
\foreach \n in {2,...,6}	{
\begin{scope}[xshift=64*\n-4]
\draw (0,0) circle (\r);
\foreach \x in {1,...,\n}	{
\draw (0,0) -- (360*\x/\n:\r);}
\end{scope}}
\end{tikzpicture}
\caption{Symmetric Radially Generated Tilings}\label{fig:symradgen}
\end{figure}

\noindent After constructing this tiling, a neighbourhood of the origin has non\--trivial intersection with each tile. The main problem on which this article is based is:

\begin{quote}
\em Can we construct monohedral tilings of the disk such that a neighbourhood of the origin has trivial intersection with at least one tile?
\end{quote}

\noindent This problem was posed in~\cite{CroftFalconerGuy}, and is in a similar vein to~\cite{Goldberg}. Informally, it may be stated in terms of slicing pizzas: can we slice a pizza into congruent pieces such that at least one piece does not touch the centre? The answer to this problem is yes, and Figure~\ref{fig:intro} displays some solutions.

\begin{figure}[htbp]\centering
\begin{tikzpicture}[scale=1.7]
\draw (0,0) circle (1);
\foreach\x in {0,...,5}	{
\draw[name path=curve] (60*\x:1) arc(60*(\x+1):60*(\x+2):1);
\path[name path=line] (60*\x+60:1) -- (60*\x:0.5);
\draw [name intersections={of=curve and line}] (60*\x+60:1) -- (intersection-1);
}
\end{tikzpicture}
\qquad
\begin{tikzpicture}[scale=1.7]
\draw (0,0) circle (1);
\foreach\x in {0,...,5}	{
\draw[name path=curve] (60*\x:1)+(60*\x+120:0.422) arc(60*(\x+1):60*(\x+2):1) -- (0,0);
\path[name path=line] (0,0) -- (60*\x+60:1);
\draw[name intersections={of=curve and line}] (60*\x+60:1) -- (intersection-1);
\draw[rotate around={60:(intersection-1)}, name intersections={of=curve and line}] (60*\x+60:1) -- (intersection-1);}
\end{tikzpicture}
\qquad
\begin{tikzpicture}[scale=1.7]
\draw (0,0) circle (1);
\foreach\x in {0,...,5}	{
\draw (60*\x:1) arc(60*(\x+1):60*(\x+2):1);
\draw (60*\x:1) arc(60*(\x+1)-30:60*(\x+2)-30:1);
}
\end{tikzpicture}
\qquad
\begin{tikzpicture}
\def\s{0.64}
\def\k{2};
\def\prad{2.65*\s};
\def\qrad{0.66*\s};
\node (p) at (-\s,0) {};
\node (q) at (\s,0) {};
\path [name path=arcp] (p) circle (\prad);
\path [name path=arcq] (q) circle (\qrad);
\path [name intersections={of=arcp and arcq, name=i, total=\t}];
\node (r) at (i-1) {};
\pgfresetboundingbox
\node [draw] at (p) [circle through=(r)] {};
\pgfextractangle{\aq}{q}{r};
\pgfextractangle{\ap}{p}{r};
\foreach \x in {0,...,5}	{
\draw (-\s,0) ++(60*\x:2*\s) ++(60*\x+\aq:\qrad) arc(60*\x+\aq:60*\x+\aq+180:\qrad/2) arc(60*\x+\aq+240:60*\x+\aq+60:\qrad/2) arc (60*\x+60+\ap:60*\x+120+\ap:\prad) arc (60*\x+\aq+120:180+60*\x+\aq+120:\qrad/2);
\draw (-\s,0) ++(60*\x:2*\s) arc(60*\x+\aq+180+30:60*\x+\aq+30:\qrad/2) arc (60*\x+60+\ap-30:60*\x+120+\ap-30:\prad) arc (60*\x+\aq+120-30:180+60*\x+\aq+120-30:\qrad/2) arc (180+60*\x+\aq+120-30+60:60*\x+\aq+120-30+60:\qrad/2);}
\end{tikzpicture}
\caption{Tilings $D_3^0$, $D_3^1$, $C_{3,2}^0$, $\widetilde{C}_{3,2}^{t*}$}\label{fig:intro}
\end{figure}
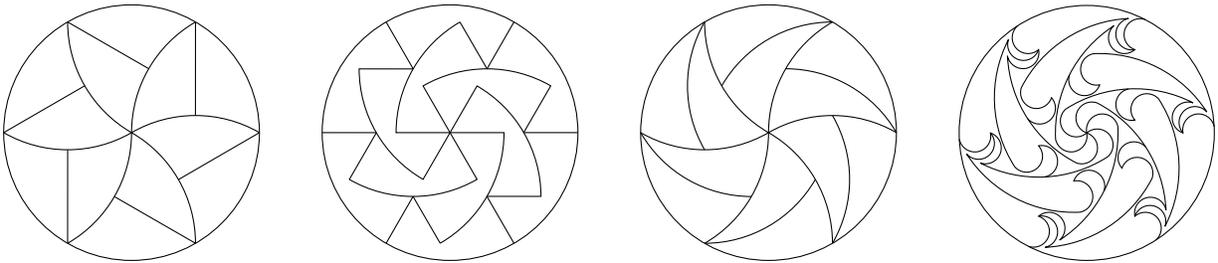

\noindent The first tiling in Figure~\ref{fig:intro}, $D_3^0$ in our notation, appears as the logo of the MASS program at Penn State~\cite{MASS}. This notation will be explained later, but note that the tilings will be presented as $C_{n,k}^t$, $\widetilde{C}_{n,k}^{t*}$ and $D_n^t$. There will be some restrictions on the indices, but $n$ and $k$ are chosen are each chosen from infinite sets of integers indicating there are infinitely many solutions; $t$ is chosen from a real interval and $t*$ is a path meaning there are uncountably many solutions. In fact, this path is required only to be continuous, simple and contained with some bounds, so even some fractal paths (i.e. the Koch snowflake) would be permissible.

We say that the tiles in Figure~\ref{fig:symradgen} are examples of tiles that are {\em radially generated} about a vertex as they consist of three components: two straight line segments, and an arc centred at the common endpoint of the straight line segments. If we fix an angle, we may think of one of these straight line segments being `dragged' by this angle about its endpoint to the other straight line segment, the locus of its other endpoint tracing out the arc component. The main idea on which this paper is based is to construct tiles that are radially generated by more than one vertex at the same time. We will show that a tile may be radially generated by at most two vertices, will describe a method for classifying such tiles, and will present ways of splitting such tiles into congruent sub\--tiles. Indeed, although no tiling in Figure~\ref{fig:intro} consists of radially generated tiles, in each case a union of tiles is a radially generated tile.

While the initial problem is interesting, it has also been answered: yes, it is possible. There are many similar questions one can ask. Does a monohedral tiling of the disk exist such that\ldots
\begin{enumerate}
\item at least one piece does not intersect the centre?
\item at least one piece does not intersect the bounding circle?
\item at least one piece does not intersect the centre and the tiling has trivial cyclic symmetric?
\item the tiling is not a symmetric radially generated tiling, and the tiling has an odd number of tiles?
\item the tiling is not a symmetric radially generated tiling, and the tiling has a line of symmetry?
\item the centre appears as the edge of a tile?
\item the centre appears in the interior of a tile?
\end{enumerate}
\noindent And the final question:
\begin{enumerate}
\item[8.] Have we provided a complete classification of monohedral disk tilings in this paper?
\end{enumerate}
It is our goal to work towards a classification of all monohedral tilings of the disk. We present a classification of such tilings obtained from the novel construction involving tiles radially generated from more than one point. After completing our classification, we will return to this list of questions, answering those we can and stating conjectures about the rest. In Section~\ref{sect:enumc} we will note the connection between one of families of tilings and the necklace numbers, and will show some combinatorial properties of our tilings.

The well\--known tiling $D_3^0$ is similar to part of a non\--periodic polygonal tiling of the plane found on page~236 of~\cite{piano}, and tiling of $D_7^0$ (a generalisation of $D_3^0$: see Section~\ref{sect:dnt}) is similar to another planar tiling found on page~515 of~\cite{shepgrun}. These images have been adapted in Figures~\ref{nonperiodic3} and~\ref{nonperiodic7}.

\begin{figure}[htbp]\centering
\begin{tikzpicture}[scale=1.8]
\draw (0:1) -- (30:1) -- (60:1) -- (90:1) -- (120:1) -- (150:1) -- (180:1) -- (210:1) -- (240:1) -- (270:1) -- (300:1) -- (330:1) -- cycle;
\foreach \x in {0,...,5}	{
\draw (0:0) -- (15+60*\x:0.52) -- (60*\x:1);}
\foreach \x in {0,...,5} {
\draw (15+60*\x:0.52) -- (60+60*\x:1);}
\end{tikzpicture}
\caption{Part of a non\--perioic planar tiling by an irregular quadrilateral}\label{nonperiodic3}
\end{figure}
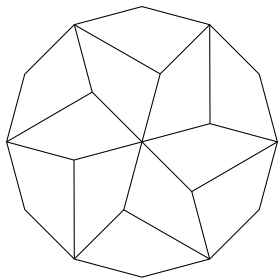

\begin{figure}[htbp]\centering
\includegraphics[height=5cm]{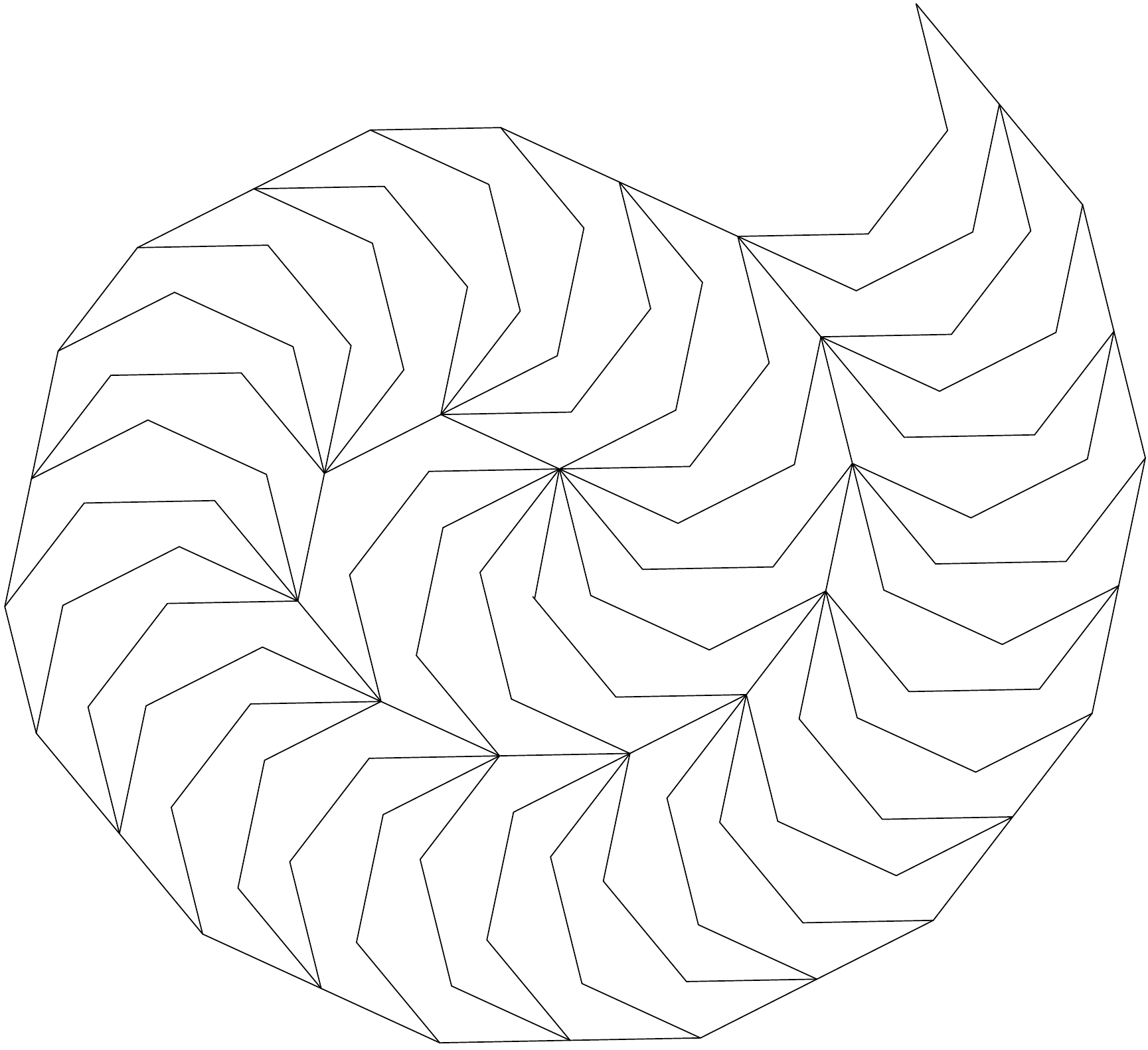}
\caption{Part of a non\--perioic planar tiling by an irregular quadrilateral}\label{nonperiodic7}
\end{figure}

\noindent We would like to acknowledge Colin Wright's observation that the tilings $D_n^0$ were generalisations of $D_3^0$, and thank him and the late Ian Porteous for many enthusiastic discussions about this problem. Thanks also to Karene Chu for the initial introduction to the problem. We also acknowledge the overwhelming and somewhat surprising media response to this research, and would like to thank all such outlets: particularly those who explained the mathematics.

\section{Monohedral Tilings}

We begin by making precise what we mean by a monohedral tiling.

\begin{definition}
Let $\gamma$ be a simple plane contour bounding the open, contractible region $U$. The closure $\overline{U}=U\cup\gamma$ is called a {\em tile}. The region $U$ is known as the {\em interior} of the tile; the region $\overline{U}^C$ the {\em exterior}; the curve $\gamma$ the {\em boundary} of the tile.
\end{definition}

\noindent We remark that according to this definition, for all points $g\in\gamma$ an arbitrarily small disk centred at $g$ has non\--trivial intersection with both the interior and exterior of the tile. That is to say, the boundary cannot simply continue to the interior to `decorate' the tile to trivially give more solutions.

A tile has a fixed position in the plane. Let $\overline{U}_1,\,\overline{U}_2$ be tiles bounded by $\gamma_1,\,\gamma_2$ respectively. Then $(\overline{U}_1\cap \overline{U}_2)\subset(\gamma_1\cap\gamma_2)$ since otherwise $U_1$ and $U_2$ would not be contractible. Hence two tiles may only intersect along their boundaries. Since tiles have fixed position, two distinct tiles may never be equal.

\begin{definition}
The tile $\overline{U}_1$ is said to be {\em congruent} to $\overline{U}_2$ if $\overline{U}_1$ is isometric to $\overline{U}_2$, or if $\overline{U}_1$ is isometric to the mirror image of $\overline{U}_2$ (i.e. the image of $\overline{U}_2$ under any orientation reversing isometry of the plane).
\end{definition}

\noindent Congruence is an equivalence relation so we may refer to all tiles in a set as being congruent to each other. A {\em copy} of a tile $\overline{U}_1$ is any congruent tile $\overline{U}_2$ with a different position in the plane.

\begin{definition}
Let $\overline{V}$ be a closed, bounded subset of the plane. If there exists a set of tiles $\mathcal{U}$ such that
\[\overline{V}=\bigcup_{\overline{U}\in\,\mathcal{U}}\overline{U}\]
then $\mathcal{U}$ is called a {\em tiling} of $\overline{V}$. If all tiles in $\mathcal{U}$ are congruent to each other, this tiling is called {\em monohedral}.
\end{definition}

\begin{definition}
We will say that a tiling of $\overline{V}_1$ is the {\em same} as a tiling of $\overline{V}_2$ if $\overline{V}_1$ and $\overline{V}_2$ are similar, and if one tiling may be mapped to the other by any orientation preserving affine transformation. Otherwise, we will say they are {\em different}.
\end{definition}

\begin{definition}
Suppose $\overline{V}$ admits different tilings by $\mathcal{U}$ and $\mathcal{U}'$, where the tiles in $\mathcal{U}'$ differ from the tiles in $\mathcal{U}$ only by their position in the plane, then $\mathcal{U}'$ will be called a {\em retiling} of $\mathcal{U}$.
\end{definition}

\begin{definition}
If for a monohedral tiling of $\overline{V}$ by tiles $\overline{U}_i$, and if the tile $\overline{U}_i$ admits a monohedral tiling into congruent tiles $\overline{W}_{i,j}$, the tiling of $\overline{V}$ by $\overline{W}_{i,j}$ is called a {\em monohedral subtiling} of $\overline{V}$.
\end{definition}

\noindent In the following sections we will fix $\overline{V}$ to be a closed disk and will describe how monohedral tilings and monohedral subtilings may arise.

\section{Radially Generated Tilings}

We will introduce the concept of radially generated tiles. If $\alpha$ is an angle, $p$ is a point plane and $P$ is any subset of the plane, then the operator $\alpha_p(P)$ will denote the anticlockwise rotation of $P$ by $\alpha$ about $p$.

\begin{definition}
Let $\eta$ be a continuous planar curve segment with no self\--intersections, let $p$ be a point on $\eta$, and let $\alpha\in(0,2\pi)$ be an angle. The union of $\eta,\,\eta'=\alpha_p(\eta)$ and arcs centred at $p$ connecting the end\--points of $\eta$ and $\eta'$ is a contour. We denote this contour by $\gamma=[\eta,p,\alpha]$, and say that $\gamma$ is {\em radially generated by $\eta$, $p$ and $\alpha$}.
\end{definition}

\begin{definition}
If $\gamma=[\eta,p,\alpha]$ is the boundary of a tile $\overline{U}$ if and only if $\gamma$ is simple. If $\overline{U}$ is a tile, we say that $\overline{U}$ is {\em radially generated by $\eta$, $p$ and $\alpha$}, denoted $\overline{U}=<\eta,p,\alpha>.$
\end{definition}

\begin{lemma}
If $\overline{U}=<\eta,p,\alpha>$ then $p$ is an end-point of $\eta$.
\end{lemma}

\begin{proof}
Suppose otherwise. Then $\gamma=[\eta,p,\alpha]$ has a self\--intersection at $p$ and is not simple.
\end{proof}

\noindent This implies that when we add circular arcs centred at $p$ to $\eta$ and $\eta'$, we add exactly one arc. We denote this arc by $\rho_p$ so that\[\gamma=[\eta,p,\alpha]=\eta\cup\eta'\cup\rho_p.\]

\begin{lemma}\label{lem:rgtile}
Let $\overline{U}=<\eta,p,2\pi/n>$ be a tile for all $n\in\Sigma\subseteq\N$. Suppose $\Sigma\neq\varnothing$ and $\min(\Sigma)=2$. Then $n$ copies of $\overline{U}$ monohedrally tile the disk.
\end{lemma}

\begin{proof}
This is equivalent to the following obvious construction. Denote $\alpha=2\pi/n$. Take a disk centred at $p$ and connect $p$ to a point on the boundary of the disk with a path $\eta$ that has no self\--intersections. Then provided that $\eta\cap(\alpha_p(\eta))$ consists only of $p$, one can repeatedly take images $\alpha_p^i(\eta)$. The induced tiling is a monohedral tiling of the disk centred at $p$.
\end{proof}

\begin{definition}
A tiling is said to be {\em radially generated} if it consists of radially generated tiles.
\end{definition}

\noindent A radially generating tiling of the disk is shown in Figure~\ref{fig:simpleradgen}

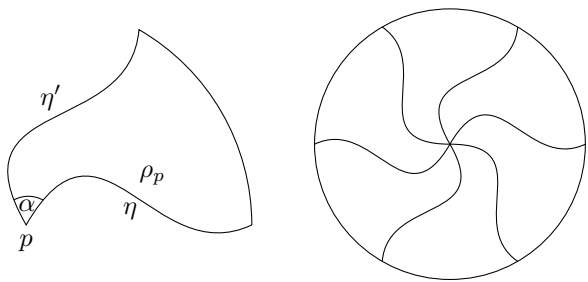
\begin{figure}[htbp]\centering
\begin{tikzpicture}
\node (b) at (-20:1.8) {};
\node (br) at (40:1.8) {};
\node (a) at (60:2) {};
\node (ar) at (120:2) {};
\draw (0,0) node[below] {$p$}
.. controls(a) and (b) .. (3, 0) node[midway,below] {$\eta$}
arc (0:60:3) node[midway,above right] {$\rho_p$}
 .. controls(br) and (ar) .. (0,0) node[midway,above] {$\eta'$}
 -- cycle;
\draw (55:0.4) arc (60:119:0.4);
\node at (0.01,0.26) {$\alpha$};
\end{tikzpicture}
\qquad
\begin{tikzpicture}[scale=0.6]
\node (a1) at (60:2) {};
\node (a2) at (120:2) {};
\node (a3) at (180:2) {};
\node (a4) at (240:2) {};
\node (a5) at (300:2) {};
\node (a6) at (0:2) {};
\node (b1) at (-20:1.8) {};
\node (b2) at (40:1.8) {};
\node (b3) at (100:1.8) {};
\node (b4) at (160:1.8) {};
\node (b5) at (220:1.8) {};
\node (b6) at (280:1.8) {};
\draw (0,0) .. controls(a1) and (b1) .. (0:3);
\draw (0,0) .. controls(a2) and (b2) .. (60:3);
\draw (0,0) .. controls(a3) and (b3) .. (120:3);
\draw (0,0) .. controls(a4) and (b4) .. (180:3);
\draw (0,0) .. controls(a5) and (b5) .. (240:3);
\draw (0,0) .. controls(a6) and (b6) .. (300:3);
\draw (0,0) circle (3);
\end{tikzpicture}
\caption{A radially generated tiling of the disk}\label{fig:simpleradgen}
\end{figure}

The trivial tiling of the disk (the disk itself) may be thought of as a degenerate radially generated tiling of the disk by the tile $\overline{U}=<p',p,2\pi>$, where $p'$ is a point distinct from $p$.

\noindent Note that if a tile $<\eta,p,\alpha>$ radially generated by a single point $p$ has a line of symmetry, then $\eta$ is necessarily a straight line segment. Moreoever if $\alpha=2\pi/n$ in this case, then $n$ copies gives the standard tiling of the disk into radial slices. We call these tilings {\em symmetric radially generated} (see Figure~\ref{fig:symradgen}).

\begin{definition}A tile will be called {\em radially generated by multiple points} if it is radially generated by at least $2$ of its vertices. I.e.,\[<\eta_1,p_1,\alpha_1>=<\eta_2,p_2,\alpha_2>=\cdots.\] We will call a tile generated by exactly 2 points a {\em wedge}; or, if it has a line of symmetry, a {\em symmetric wedge}.\end{definition}

\noindent If the rotation angles at each vertex are of the form $2\pi/n$ for some integer $n$ then the wedge may tile the disk about either vertex as in Figure~\ref{fig:wedgex2}.

\begin{figure}[htbp]\centering
\begin{tikzpicture}
\def\s{0.5}
\coordinate (p) at (-\s,0);
\coordinate (q) at (\s,0);
\def\prad{3.34*\s};
\def\qrad{1.42*\s};
\path [name path=arcp] (p) circle (\prad);
\path [name path=arcq] (q) circle (\qrad);
\path [name intersections={of=arcp and arcq, name=i, total=\t}];
\coordinate (r) at (i-2);
\pgfextractangle{\ap}{p}{r};
\pgfextractangle{\aq}{q}{r};
\pgfresetboundingbox;
\draw (p) circle (\prad);
\foreach \x in {1,...,6}	{
\draw[rotate around={60*\x:(p)}] ([rotate around={60*\x:(p)}]r) -- ([rotate around={60*\x:(p)}]q) -- ++(60+\aq:\qrad) arc(60+\ap:120+\ap:\prad) -- (p);}
\end{tikzpicture}
\quad
\begin{tikzpicture}
\def\s{0.5}
\coordinate (p) at (-\s,0);
\coordinate (q) at (\s,0);
\def\prad{3.34*\s};
\def\qrad{1.42*\s};
\path [name path=arcp] (p) circle (\prad);
\path [name path=arcq] (q) circle (\qrad);
\path [name intersections={of=arcp and arcq, name=i, total=\t}];
\coordinate (r) at (i-2);
\pgfextractangle{\ap}{p}{r};
\pgfextractangle{\aq}{q}{r};
\pgfresetboundingbox;
\draw (q) circle (\prad);
\foreach \x in {1,...,6}	{
\draw[rotate around={60*\x:(q)}] (q) -- ++(60+\aq:\qrad) arc(60+\ap:120+\ap:\prad) -- ([rotate around={60*\x:(q)}]p) -- ++(180+\aq:\qrad);}
\end{tikzpicture}
\caption{A wedge tiling the disk in two ways}\label{fig:wedgex2}
\end{figure}
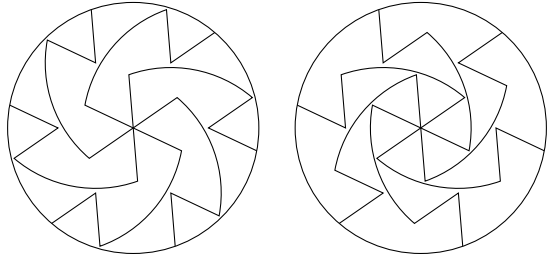

\begin{theorem}\label{thm:equalangles}
If a wedge is given by $<\eta,p,\alpha>=<\nu,q,\beta>$, and this wedge tiles the disk around either $p$ or $q$, then $\alpha=\beta=\pi/n,$ where $n\ge3$ is an odd integer.
\end{theorem}

\begin{proof}
Let $\rho_p$ and $\rho_q$ denote the arc components of the boundary centred at $p$ and $q$ respectively. Since $q$ is a vertex of the boundary of the wedge, its radially generated image $\alpha_p(q)\in\eta'=\alpha_p(\eta)$ is also a vertex of the boundary of the wedge. Hence its image $\beta_q\circ\alpha_p(q)\in\eta$ is another vertex of the boundary wedge (since the wedge is radially generated about $q$). Repeatedly taking the image of $q$ under alternating compositions of rotations $\alpha_p$ and $\beta_q$, the image alternates from belonging to $\eta$ or $\eta'$, and this process necessarily terminates when the image is $p\in\eta,\eta'$. Although we could in principle keep taking images after this, only those images obtained up to and including $p$ are vertices of the wedge. Notice that all the images of the composite function $(\beta_q\circ\alpha_p)^j(q)\in\eta$, and that $\beta_q\circ\alpha_p$ is a rotation by $\alpha+\beta$. Since $p$ is one of the images of $(\beta_q\circ\alpha_p)^j(q)$, and the images are evenly spaced around the fixed point of the map $\beta_q\circ\alpha_p$, we can conclude that the fixed point of $\beta_q\circ\alpha_p$ is equidistant from $p$ and $q$. Notice that since $q$ is fixed by $\beta_q$, we may apply the same argument to the images $(\alpha_p\circ \beta_q)^j(q)\in\eta'$ so that the fixed point of $\alpha_p\circ\beta_q$ is also equidistant from $p$ and $q$.

Suppose that the plane has a complex structure and, without loss of generality, fix $p=-1$ and $q=1$ so that we may describe our rotations as the complex rotations:
\begin{align*}
	\alpha_p	&:	z\mapsto -1+e^{\alpha i}(z+1)\\
	\beta_q	    &:	z\mapsto 1+e^{\beta i}(z-1).
\end{align*}
Let $c_+$ denote the fixed point of $\beta_q\circ\alpha_p$, and $c_-$ denote the fixed point of $\alpha_p\circ\beta_q$. The preceding argument tells us that $\Re(c_+)=\Re(c_-)=0$. Using this, again with the fact that $\alpha_p\circ\beta_q(q)=\alpha_p(q)$ (Figure~\ref{a_equals_b}) we see that the triangle $pqc_+$ is isosceles, and the figure presented is symmetric about the line through $p$ and $c_+$. Hence $\alpha=\beta$.
\begin{figure}[htbp]\centering
\begin{tikzpicture}[scale=2]
\node[circle, draw, inner sep = 0,minimum size=4pt, fill=black, label=below:$p$] (p) at (-1,0) {};
\node[circle, draw, inner sep = 0,minimum size=4pt, fill=black, label=below:$q$] (q) at (1,0) {};
\node[circle, draw, inner sep = 0,minimum size=4pt, fill=black, label=below:$c_+$] (c) at (0,0.32) {};
\node[circle, draw, inner sep = 0,minimum size=4pt, fill=black, label={right:$\alpha_q\circ\alpha_p(q)=\alpha_p(q)$}] (aq) at (0.62,1.18) {};
\draw (p) -- (q);
\draw (p) -- (aq);
\draw (c) -- (q);
\draw (c) -- (aq);
\draw (q) arc (0:36:2);
\draw[draw=black!50] (p) -- (c);
\draw[draw=black!50] (c) -- (0.9,0.62);
\node at (-0.8,0.07) {$\alpha$};
\node at (0.3,0.37) {$\alpha+\beta$};
\end{tikzpicture}
\caption{Proof that $\alpha=\beta$.}\label{a_equals_b}
\end{figure}
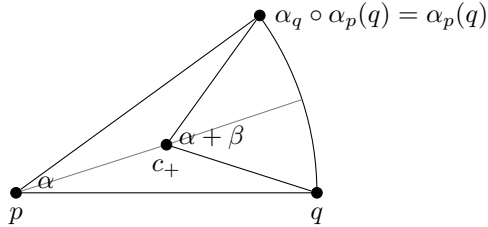

\noindent So the map $\alpha_p\circ \alpha_q$ (hence also $\alpha_q\circ \alpha_p$) is a rotation by $2\alpha$ and therefore has order $n=2\pi/2\alpha$, hence $\alpha=\pi/n$ for some integer $n\ge2$. It can be shown that the centres of rotation are given by
\[c_\pm=\pm\frac{\sin(\pi/n)}{\cos(\pi/n)+1}i.\]
All of the images $(\alpha_q\circ\alpha_p)^j(q)$, including those that are not vertices of the wedge, are the vertices of a regular $n$\--gon centred at $c_+$. Similarly, all images $(\alpha_p\circ\alpha_q)^j(q)$ are the vertices of a regular $n$\--gon centred at $c_-$. This configuration has the real axis as a line of symmetry, hence the configuration for $\Im(z)\ge0$ and the configuration for $\Im(z)\le0$ are each the configuration of vertices of the wedge. Since there are necessarily an odd number of images of $q$ on the wedge (the same number on $\eta$ as $\eta'$, remembering that $p$ is the same vertex on both), then $n\ge3$ is odd. The configuration of the vertices of the two overlapping regular $n$\--gons is the same as the configuration of two touching regular reflex $n$\--gons.
\end{proof}

\begin{cor}
If a wedge is given by $<\eta,p,\alpha>=<\nu,q,\alpha>$ tiles the disk around $p$ or $q$,  then the arc components of the radial generation $\rho_p$ and $\rho_q$ have the same curvature.
\end{cor}

\begin{proof}
From Theorem~\ref{thm:equalangles}, we have $\alpha=\pi/n$ for some odd $n\ge3$. Hence $2n$ copies of the wedge tile the disk. Since the area of the wedge doesn't change according to which point we tile the disk around, neither does the curvature of the circle bounding the disk, which is made up of $2n$ copies of $\rho_p$ or $\rho_q$.
\end{proof}

\begin{definition}
A wedge $<\eta,p,\pi/n>=<\nu,q,\pi/n>$ where $n\ge3$ is an odd integer will be called an {\em $n$\--wedge}.
\end{definition}

\begin{theorem}\label{thm:uncountable}
For any odd integer $n\ge3$, and fixed $p$, $q$ there are uncountably many $n$\--wedges of the form $\overline{U}=<\eta,p,\alpha>=<\nu,q,\alpha>$ (where $\alpha=\pi/n$) that tile the disk.
\end{theorem}

\begin{proof}
Suppose $r_p$ and $\alpha_p(r_p)$ are the endpoints of the arc $\rho_p$, and $r_q$ and $\alpha^{-1}_q(r_q)$ are the endpoints of $\rho_q$, and again impose a complex structure with $p=-1,\,q=1$ . Then we necessarily have $|r_p|\ge1$ since otherwise radially tiling about $p$ would produce a disk that did not contain $q$. Suppose (at least for now) that $r_p\neq q$, and that $q$ and $r_p$ are connected by a straight line segment. Then each image $(\alpha_q\circ\alpha_p)^j(q)$ is connected to the corresponding image of $(\alpha_q\circ\alpha_p)^j(r_p)$ via the image under the same map of the straight line segment connecting $q$ and $r_p$. At each of the $n$ images of $q$, there are two images of this straight line segment which meet at angle of $\pi/n$. We call any such pair a {\em groove}, and say the {\em length} of the groove is $|r_p-q|$.

\noindent Recall that the images of $q$ under the alternating maps $\alpha_p,\,\alpha_q$ form the vertices of a regular reflex $n$\--gon, with the final image being $p$. Hence
\[(\alpha_q\circ\alpha_p)^{(n-1)/2}(q)=p,\]
so that the groove at $q$ is rotated by an angle of $(2\pi/n)(n-1)/2=\pi(n-1)/n$ to give the groove at $p$. Since the internal angle of a groove is $\pi/n$, and $\pi(n-1)/n+\pi/n=\pi$, the outer components of the grooves (i.e. the straight line segments connecting $q$ to $r_p$ and $p$ to $r_q$) are parallel. Moreover, $|r_p-q|=|r_q-p|$, so that $|r_p|=|r_q|$.

So the position of $r_p$ is always $\pi_0(r_q)$, i.e. $r_p$ is the image of $r_q$ under rotation by $\pi$ about the origin. Hence the locus $\mathcal{R}_p$ of $r_p$ for which $\overline{U}$ is a tile (and hence a wedge) is isometric to the corresponding locus $\mathcal{R}_q$ for $r_q$, and one is obtained from the other by rotation of $\pi$ about the origin. Since the configuration of vertices is symmetric about the imaginary axis, then the loci are symmetric about the imaginary axis. Hence each locus is symmetric about the real axis.

According to the construction, $\overline{U}$ is a wedge if and only if it is a tile. It fails to be a tile if and only if $r_q$ is chosen such that the boundary of $\overline{U}$ has a self intersection. There are three ways in which this may happen.
\begin{enumerate}[(1)]
\item If the groove at $q$ intersects $\rho_p$ at a point other than $r$. Since both components of the groove have the same length, and both $p$ and $q$ lie on the real axis, this is equivalent to the groove having the real axis as its angle bisector. Since the internal angle of the groove is $\pi/n$, the critical locus for $r_q$ in this case contains the ray emanating from $q$ with angle $-\pi/2n$.
\item We can make a similar argument about the intersection of the groove at $p$ with $\rho_q$. However, we may just argue by symmetry: since the locus is symmetric about the real axis, the critical locus contains the ray emanating from $q$ with angle $\pi/2n$.
\item Finally, it can happen that some image of $\rho_p$ may intersect an image of $q$. Since we are only interested in the `smallest' value of $r_q$ that makes this happen (as $r_q$ changes maybe there will be more intersections, but even the first intersection is enough to mean $\overline{U}$ is not a wedge) we look for a condition on $r_q$ such that $\alpha_q(\rho_p)\cap\alpha_p(q)$ is non\--empty. Note that $\alpha_p(q)$ is fixed, whereas $\alpha_q(\rho_p)$ depends on $r_q$ since $r_q$ is an endpoint of $\rho_p$. Hence the image of this component of the critical locus under $\alpha_q$ is $\alpha_q(\rho_p)$ where $\rho_q$ is chosen such that there is a single intersection with $\alpha_p(q)$. The centre of curvature for $\alpha_q(\rho_p)$ is $\alpha_q(p)$, hence the radius of the arc through $\alpha_p(q)$ is
\begin{align*}
R	&=	|\alpha_q(p)-\alpha_p(q)|\\
	&=	2|1-2e^{\pi i/n}|\\
	&=	2(1+\sqrt{5-4\cos(\pi/n)}).
\end{align*}
So this component of the critical locus for  $r_q$ is given by $|p-r_q|=R$ (recall that $p,\,q$ were fixed so the value of $R$ would need to be rescaled for other positions of $p,\,q$).
\end{enumerate}
The critical loci bound an open region $\mathcal{R}_q$ whose closure contains $q$, and for all $r_q\in\mathcal{R}_q\cup{q},$ $\overline{U}$ is a wedge. Even though $\mathcal{R}_q$ gets smaller as $n$ gets larger, $\mathcal{R}_q$ is never empty while $n<\infty$. Hence for each fixed odd $n\ge3$, there are uncountably many wedges (Figure~ref{fig:critloc}).
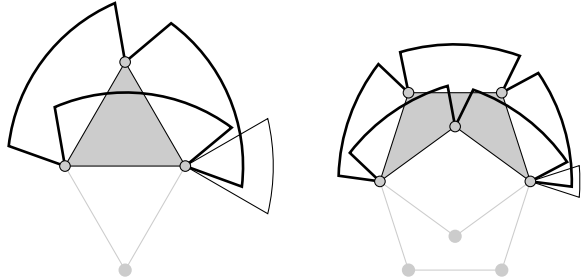
\begin{figure}[htbp]\centering
\begin{tikzpicture}[scale=0.8]
\coordinate (p) at (-1,0);
\coordinate (q) at (1,0);
\coordinate (r) at (0,1.73);
\coordinate (s) at (0,-1.73);
\coordinate (t-1-1) at (2,-0.36);
\draw[draw=black!20] (p) -- (s) -- (q);
\filldraw[fill=black!20] (p) -- (q) -- (r) -- (p) -- cycle;
\draw[shape=circle] (q) -- (2.37,0.79) arc (13.22:-13.22:3.46) -- (q) -- cycle;
\draw[line width=1pt] (q) -- ++(-20:1) arc (-6.6:53.4:2.96) -- (r) -- (-0.17,2.71) arc (113.4:173.4:2.96) -- (p) -- (-1.17,0.98) arc (113.4:53.4:2.96) -- (q);
\node[shape=circle,draw,fill=black!20,inner sep = 0,minimum size=4pt] at (p) {};
\node[shape=circle,draw,fill=black!20,inner sep = 0,minimum size=4pt] at (q) {};
\node[shape=circle,draw,fill=black!20,inner sep = 0,minimum size=4pt] at (r) {};
\node[shape=circle,draw=black!20,fill=black!20,line width=1pt,inner sep = 0,minimum size=4pt] at (s) {};
\end{tikzpicture}
\qquad
\begin{tikzpicture}[scale=1]
\coordinate (p) at (-1,0);
\coordinate (q) at (1,0);
\coordinate (r-1) at (0.62,1.18);
\coordinate (r-2) at (-0.62,1.18);
\coordinate (r-3) at (0.62,-1.18);
\coordinate (r-4) at (-0.62,-1.18);
\coordinate (s-1) at (0,0.73);
\coordinate (s-2) at (0,-0.73);
\draw[draw=black!20] (p) -- (r-4) -- (r-3) -- (q) -- (s-2) -- (p) -- cycle;
\filldraw[fill=black!20] (p) -- (r-2) -- (r-1) -- (q) -- (s-1) -- (p) -- cycle;
\draw[shape=circle] (q) -- (1.65,0.21) arc (4.545:-4.545:2.66) -- (q) -- cycle;
\draw[line width=1pt] (q) -- (1.55,-0.07) arc (-1.6:34.4:2.55) -- (r-1) -- (0.86,1.67) arc (70.04:106.4:2.55) -- (r-2) -- (-1.02,1.56) arc (142.4:178.4:2.55) -- (p) -- (-1.4,0.38) arc (142.4:106.4:2.55) -- (s-1) -- (0.24,1.22) arc (70.4:34.4:2.55) -- (q) -- cycle;
\node[shape=circle,draw,fill=black!20,inner sep = 0,minimum size=4pt] at (p) {};
\node[shape=circle,draw,fill=black!20,inner sep = 0,minimum size=4pt] at (q) {};
\node[shape=circle,draw,fill=black!20,inner sep = 0,minimum size=4pt] at (r-1) {};
\node[shape=circle,draw,fill=black!20,inner sep = 0,minimum size=4pt] at (r-2) {};
\node[shape=circle,draw,fill=black!20,inner sep = 0,minimum size=4pt] at (s-1) {};
\node[shape=circle,draw=black!20,fill=black!20,line width=1pt,inner sep = 0,minimum size=4pt] at (s-2) {};
\node[shape=circle,draw=black!20,fill=black!20,line width=1pt,inner sep = 0,minimum size=4pt] at (r-3) {};
\node[shape=circle,draw=black!20,fill=black!20,line width=1pt,inner sep = 0,minimum size=4pt] at (r-4) {};
\end{tikzpicture}
\caption{The vertex configurations and critical loci for $n=3$ and $n=5$.}\label{fig:critloc}
\end{figure}
\end{proof}

\noindent Note that taking $r_q=q$ defines an $n$\--wedge without grooves (i.e. with groove length 0). Such a wedge is always symmetric.

The groove length was defined as $|r_q-q|$, not as the length of the straight line segment connecting $r_q$ to $q$. This is because unless we require our wedge to be symmetric, the groove need not be constructed from straight line segments. They can be any simple path segments such that neither it nor any of its images intersect with the circle bounding the tiled disk, or any other tiles in the tiling.

For a symmetric $n$\--wedge the line through $r_p$ and $r_q$ is colinear with the line through $p$ and $q$, so in this case $\mathcal{R}_q$ is an interval rather than a region.

\begin{cor}\label{prop:no3wedge}
A tile radially generated by 3 vertices that tiles the disk does not exist.
\end{cor}

\begin{proof}
Any two generating vertices of wedge have images in the configuration of a regular reflex $n$\--gon. Adding a new generating vertex would add a similar configuration between itself and each of the original vertices, and the overall configuration would not be a regular reflex $n$\--gon for $n\ge5$. For $n=3$, it is possible to satisfy the vertex configuration (arranging the vertices in an equilateral triangle) but not the arc configuration): each vertex is opposite an arc so the resulting shape would be convex.
\end{proof}

\begin{cor}\label{cor:2}
A symmetric wedge has one only line of symmetry.
\end{cor}

\begin{proof}
In this case the vertices about which the tile is radially generated are mirror images of each other in a line of symmetry. If there were another line of symmetry, it would imply more vertices about which the tile may be radially generated, contradicting Corollary~\ref{prop:no3wedge}.
\end{proof}

\section{Families of subtilings}

In the previous section we introduced the concept of a wedge, and classified all types of wedge that may be used to tile disks. In this section we will present two known ways of subtiling a wedge to produce families of monohedral tilings of the disk.

\begin{definition}
A {\em family of disk tilings} is an equivalence class of tilings up to retiling or scaling.
\end{definition}

\noindent So a family of monohedral disk tilings may be described in terms of its fundamental tile.

\subsection{$D_n^t$}\label{sect:dnt}

\begin{prop}\label{d:prop}
Let $\overline{U}$ be a symmetric $n$\--wedge with $n\ge3$ odd. Then $\overline{U}$ may be tiled by two congruent tiles.
\end{prop}

\begin{proof}
These tiles are just the left\-- and right\--handed parts of $\overline{U}$ obtained by splitting $\overline{U}$ in half through its line of symmetry.
\end{proof}

\noindent Suppose $\overline{U}$ is an $n$\--wedge with groove length $t$ (recall that the locus $\mathcal{R}_q$ for $r_q$ is an interval since $\overline{U}$ is symmetric), and suppose scaling is such that $\sup|r_q-q|=1$. We denote the family of tilings obtained from this tile by $D_n^t$, where $n\ge3$ is an odd integer and $t\in[0,1)$. Any tiling in $D_n^t$ contains $4n$ tiles, and $|D_n^t|$=2; that is to say, this collection of tiles may tile the disk in two different ways (Figure~\ref{fig:dnt}).

\begin{figure}[htbp]\centering
\begin{tikzpicture}[scale=1.8]
\draw (0,0) circle (1);
\foreach\x in {0,...,5}	{
	\draw[name path=curve] (60*\x:1) arc(60*(\x+1):60*(\x+2):1);
	\path[name path=line] (60*\x+60:1) -- (60*\x:0.5);
	\draw [name intersections={of=curve and line}] (60*\x+60:1) -- (intersection-1);
}
\end{tikzpicture}
\qquad
\begin{tikzpicture}[scale=1.8]
\draw (0,0) circle (1);
\def\t{0.1}
\foreach \x in {0,...,6}	{
\draw (60*\x:1) -- ++(180+60*\x:\t) node (c-\x-1) {} -- ++(60+60*\x:\t) arc(60+60*\x:90+60*\x:1) node (c-\x-2) {} arc(90+60*\x:120+60*\x:1) -- (0,0);}
\foreach \x in {0,...,6}	{
\draw[rotate=60*\x] ([rotate=60*\x]c-1-1.center) -- ([rotate=60*\x]c-0-2.center);}
\end{tikzpicture}
\qquad
\begin{tikzpicture}[scale=1.8]
\draw (0,0) circle (1);
\def\t{0.35}
\foreach \x in {0,...,6}	{
\draw (60*\x:1) -- ++(180+60*\x:\t) node (c-\x-1) {} -- ++(60+60*\x:\t) arc(60+60*\x:90+60*\x:1) node (c-\x-2) {} arc(90+60*\x:120+60*\x:1) -- (0,0);}
\foreach \x in {0,...,6}	{
\draw[rotate=60*\x] ([rotate=60*\x]c-1-1.center) -- ([rotate=60*\x]c-0-2.center);}
\end{tikzpicture}
\\\medskip\medskip
\begin{tikzpicture}[scale=1.8]
\draw (0,0) circle (1);
\foreach \x in {0,...,9}	{
\draw (36*\x:1) arc (36+36*\x:72+36*\x-18:1) node (c-\x-1) {} arc (36+36*\x+18:72+36*\x:1) node (c-\x-2) {} arc (108+36*\x:144+36*\x:1);}
\foreach \x in {0,...,9}	{
\draw[rotate=36*\x] ([rotate=36*\x]c-1-1.center) -- ([rotate=36*\x]c-0-2.center);}
\end{tikzpicture}
\qquad
\begin{tikzpicture}[scale=1.8]
\draw (0,0) circle (1);
\def\t{0.1}
\foreach \x in {0,...,9}	{
\draw (36*\x:1) -- ++(180+36*\x:\t) -- ++(36+36*\x:\t) arc(36+36*\x:72+36*\x-18:1) node (c-\x-1) {} arc(36+36*\x+18:72+36*\x:1) -- ++(-108+36*\x:\t) node (c-\x-2) {} -- ++(108+36*\x:\t) arc (108+36*\x:144+36*\x:1) -- (0,0);}
\foreach \x in {0,...,9}	{
\draw[rotate=36*\x] ([rotate=36*\x]c-1-1.center) -- ([rotate=36*\x]c-0-2.center);}
\end{tikzpicture}
\qquad
\begin{tikzpicture}[scale=1.8]
\draw (0,0) circle (1);
\def\t{0.22}
\foreach \x in {0,...,9}	{
\draw (36*\x:1) -- ++(180+36*\x:\t) -- ++(36+36*\x:\t) arc(36+36*\x:72+36*\x-18:1) node (c-\x-1) {} arc(36+36*\x+18:72+36*\x:1) -- ++(-108+36*\x:\t) node (c-\x-2) {} -- ++(108+36*\x:\t) arc (108+36*\x:144+36*\x:1) -- (0,0);}
\foreach \x in {0,...,9}	{
\draw[rotate=36*\x] ([rotate=36*\x]c-1-1.center) -- ([rotate=36*\x]c-0-2.center);}
\end{tikzpicture}
\caption{$D_n^t$}\label{fig:dnt}
\end{figure}

Exceptionally, increasing the groove length of a symmetric $3$\--wedge to the critical value $t=1$ similarly splits the symmetric $3$\--wedge into congruent pieces with opposite orientations. We call this family $D_3^1$ (Figure~\ref{fig:d31}) as it shares the same combinatorial properties as other tilings of type $D_n^t$.

\begin{figure}[htbp]\centering
\begin{tikzpicture}[scale=1.8]
\draw (0,0) circle (1);
\foreach\x in {0,...,5}	{
\draw[name path=curve] (60*\x:1)+(60*\x+120:0.422) arc(60*(\x+1):60*(\x+2):1) -- (0,0);
\path[name path=line] (0,0) -- (60*\x+60:1);
\draw[name intersections={of=curve and line}] (60*\x+60:1) -- (intersection-1);
\draw[rotate around={60:(intersection-1)}, name intersections={of=curve and line}] (60*\x+60:1) -- (intersection-1);}
\end{tikzpicture}
\caption{$D_3^1$}\label{fig:d31}
\end{figure}

\subsection{$\widetilde{C}_{n,k}^{t*}$}

\begin{prop}\label{c:prop}
Let $\overline{U}$ be an $n$\--wedge. Then, provided the groove is sufficiently bounded, $\overline{U}$ may be tiled by $k$ congruent tiles for any positive integer $k$.
\end{prop}

\begin{proof}
Let $\overline{U}=<p,\eta,\alpha>$. That is, $\overline{U}=\eta\cup\alpha_p(\eta)\cup\rho_p$. Let $\alpha'=\alpha/k$. Then, provided the groove length is sufficiently small to avoid intersections, $k$ copies of $\overline{U}_k=<p,\eta,\alpha'>$ tile $\overline{U}$.
\end{proof}

\noindent Since $\overline{U}_k$ is radially generated about $p$ by an angle $\pi/nk$, $nk$ copies of $\overline{U}_k$ tile the disk about $p$ according to Lemma~\ref{lem:rgtile}. Suppose $\overline{U}$ is also generated about point $q$. Since $k$ copies of $\overline{U}_k$ tile $\overline{U}$, one may tile the disk about $q$ with unions of $k$ copies of $\overline{U}_k$.

In the case that the $n$\--wedge $\overline{U}$ is not symmetric, the family of tilings obtained from this tile will be denoted $\widetilde{C}_{n,k}^{t*}$ where $n\ge3$ is an odd integer, $k\ge1$ is an integer, and $t\in[0,1)$. We use the notation $t*$ to indicate that both the length and the path of the groove have been fixed. As stated, each tiling consists of $2nk$ tiles. We have $|\widetilde{C}_{n,k}^{t*}|=4$ since we can tile about either vertex of the $n$\--wedge obtained as the union of $k$ wedges, and for each of these tilings we can take the tiling with the opposite orientation. Note that since these wedges are not required to be symmetric the grooves may be `decorated' in any way, provided it is cyclically consistent and doesn't produce any intersections. In a similar way to Theorem~\ref{thm:uncountable}, there are uncountably many $t*$ for each $n,\,k$ as the corresponding critical locus is never empty. However, it becomes so small that with straight edges grooves it would be practically difficult to distinguish between illustrations of $C_{n,k}^t$ and $\widetilde{C}_{n,k}^{t*}$ for $n\ge5$, so in Figure~\ref{fig:c3kt} we present only tilings of type $\widetilde{C}_{3,k}^{t*}$.

\begin{figure}[htbp]\centering
\begin{tikzpicture}
\def\s{0.8}
\def\k{2};
\def\prad{2.5*\s};
\def\qrad{0.54*\s};
\node (p) at (-\s,0) {};
\node (q) at (\s,0) {};
\path [name path=arcp] (p) circle (\prad);
\path [name path=arcq] (q) circle (\qrad);
\path [name intersections={of=arcp and arcq, name=i, total=\t}];
\node (r) at (i-1) {};
\pgfresetboundingbox
\node [draw] at (p) [circle through=(r)] {};
\pgfextractangle{\aq}{q}{r};
\pgfextractangle{\ap}{p}{r};
\foreach \x in {0,...,5}	{
\draw (-\s,0) ++(60*\x:2*\s) ++(60*\x+\aq:\qrad) -- ++(180+60*\x+\aq:\qrad) -- ++(60+60*\x+\aq:\qrad) arc (60*\x+60+\ap:60*\x+120+\ap:\prad) -- (-\s,0);
\foreach \y in {2,...,\k}	{
\draw (-\s,0) ++(60*\x:2*\s) -- ++(60*\x+\y*60/\k-60/\k+\aq:\qrad) arc (60*\x+60+\ap+60/\k*\y-60/\k-60:60*\x+120+\ap+60/\k*\y-60/\k-60:\prad) -- ++(60*\x+60+\y*60/\k-60/\k+\aq:-\qrad) -- ++(60*\x+60+\y*60/\k-60/\k+\aq+60:\qrad);}
}
\end{tikzpicture}
\qquad
\begin{tikzpicture}
\def\s{0.8}
\def\k{3};
\def\prad{2.5*\s};
\def\qrad{0.54*\s};
\node (p) at (-\s,0) {};
\node (q) at (\s,0) {};
\path [name path=arcp] (p) circle (\prad);
\path [name path=arcq] (q) circle (\qrad);
\path [name intersections={of=arcp and arcq, name=i, total=\t}];
\node (r) at (i-1) {};
\pgfresetboundingbox
\node [draw] at (p) [circle through=(r)] {};
\pgfextractangle{\aq}{q}{r};
\pgfextractangle{\ap}{p}{r};
\foreach \x in {0,...,5}	{
\draw (-\s,0) ++(60*\x:2*\s) ++(60*\x+\aq:\qrad) -- ++(180+60*\x+\aq:\qrad) -- ++(60+60*\x+\aq:\qrad) arc (60*\x+60+\ap:60*\x+120+\ap:\prad) -- (-\s,0);
\foreach \y in {2,...,\k}	{
\draw (-\s,0) ++(60*\x:2*\s) -- ++(60*\x+\y*60/\k-60/\k+\aq:\qrad) arc (60*\x+60+\ap+60/\k*\y-60/\k-60:60*\x+120+\ap+60/\k*\y-60/\k-60:\prad) -- ++(60*\x+60+\y*60/\k-60/\k+\aq:-\qrad) -- ++(60*\x+60+\y*60/\k-60/\k+\aq+60:\qrad);}
}
\end{tikzpicture}
\qquad
\begin{tikzpicture}
\def\s{0.75}
\def\k{2};
\def\prad{2.65*\s};
\def\qrad{0.66*\s};
\node (p) at (-\s,0) {};
\node (q) at (\s,0) {};
\path [name path=arcp] (p) circle (\prad);
\path [name path=arcq] (q) circle (\qrad);
\path [name intersections={of=arcp and arcq, name=i, total=\t}];
\node (r) at (i-1) {};
\pgfresetboundingbox
\node [draw] at (p) [circle through=(r)] {};
\pgfextractangle{\aq}{q}{r};
\pgfextractangle{\ap}{p}{r};
\foreach \x in {0,...,5}	{
\draw (-\s,0) ++(60*\x:2*\s) ++(60*\x+\aq:\qrad) arc(60*\x+\aq:60*\x+\aq+180:\qrad/2) arc(60*\x+\aq+240:60*\x+\aq+60:\qrad/2) arc (60*\x+60+\ap:60*\x+120+\ap:\prad) arc (60*\x+\aq+120:180+60*\x+\aq+120:\qrad/2);
\draw (-\s,0) ++(60*\x:2*\s) arc(60*\x+\aq+180+30:60*\x+\aq+30:\qrad/2) arc (60*\x+60+\ap-30:60*\x+120+\ap-30:\prad) arc (60*\x+\aq+120-30:180+60*\x+\aq+120-30:\qrad/2) arc (180+60*\x+\aq+120-30+60:60*\x+\aq+120-30+60:\qrad/2);}
\end{tikzpicture}
\caption{$\widetilde{C}_{3,k}^{t*}$}\label{fig:c3kt}
\end{figure}

\subsection{$C_{n,k}^t$}

The family $C_{n,k}^t$ is the special case of $\widetilde{C}_{n,k}^{t*}$ based on a symmetric $n$\--wedge. The construction is based on Proposition~\ref{c:prop}, but this time we require the underlying $n$\--wedge to be symmetric. The families are considered separately as they have different combinatorial properties. Tilings in $C_{n,k}^t$ still contain $2nk$ tiles, but $|C_{n,k}^t|$ is more complicated and will be discussed in Section~\ref{sect:enumc}. Recall that since the wedges here are symmetric, the grooves necessarily consist of straight edges (Figure~\ref{fig:cnkt}).

\begin{figure}[htbp]\centering
\CNK{3}{2}{1.8}\qquad\CNK{3}{3}{1.8}\qquad\CNK{5}{2}{1.8}
\\\medskip\medskip
\CNKTa{2}{0.23}{1.8}\qquad\CNKTa{3}{0.2}{1.8}\qquad
\begin{tikzpicture}[scale=1.8]
\draw (0,0) circle (1);
\def\t{0.13}
\def\n{5}
\foreach \x in {1,...,10}	{
\draw (\x*180/\n:1) -- ++(\x*180/\n+180:\t) -- ++(\x*180/\n+180/\n:\t) arc (\x*180/\n+180/\n:\x*180/\n+2*180/\n:1) -- ++(\x*180/\n+180+2*180/\n:\t) -- ++(\x*180/\n+3*180/\n:\t) arc (\x*180/\n+3*180/\n:\x*180/\n+4*180/\n:1) -- (0,0);
\draw (\x*180/\n:1-\t) -- ++(\x*180/\n+18:\t) arc (\x*180/\n+180/\n-18:\x*180/\n+2*180/\n-18:1) -- ++(\x*180/\n+180+2*180/\n-18:\t) -- ++(\x*180/\n+3*180/\n-18:\t) arc (\x*180/\n+3*180/\n-18:\x*180/\n+4*180/\n-18:1) -- ++(\x*180/\n+9*180/\n-18:\t) -- ++(\x*180/\n+5*180/\n-18:\t);}
\end{tikzpicture}
\caption{$C_{n,k}^t$}\label{fig:cnkt}
\end{figure}

\subsection{Subtilings}\label{sect:subtilings}

The monohedral tilings we have classified are summarised in the following subtiling diagram where an arrow $X\to Y$ indicates that there exists a tiling in the family $X$ that is necessarily a subtiling of family member of $Y$; and $X\xhookrightarrow{} Y$ indicates that the tiling family $X$ is a special case of the tiling family $Y$ (i.e. when the tiles of $Y$ are necessarily symmetric but those of $X$ aren't). This diagram is transitive, but no arrow is drawn when there may be a subtiling. E.g. some symmetric radially generated tilings have symmetric radially generated subtilings (e.g. standard radial cut into 6 pieces is a subtiling of a standard radial cut into 3 pieces), while others don't.


\begin{center}
\includegraphics{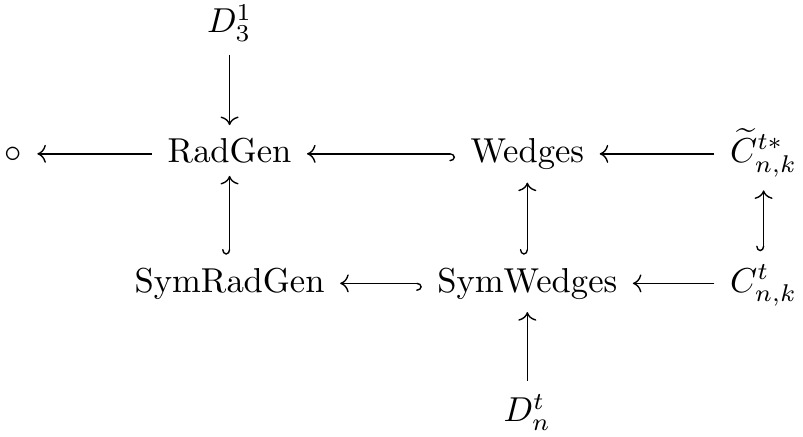}
\end{center}

\noindent Notice that even though $D_3^1$ appeared as an exceptional case when classifying $D_n^t$, it is not a subtiling of a tiling containing symmetric wedges. Its only non\--trivial subtilings are by radially generated tiles (Figure~\ref{fig:d31radgen}).
\begin{figure}[htbp]\centering
\begin{tikzpicture}[scale=1.8]
\draw (0,0) circle (1);
\foreach\x in {0,...,5}	{
\draw[name path=curve] (60*\x:1)+(60*\x+120:0.422) arc(60*(\x+1):60*(\x+2):1) -- (0,0);
\path[name path=line] (0,0) -- (60*\x+60:1);
\draw[name intersections={of=curve and line}] (60*\x+60:1) -- (intersection-1);
\draw[rotate around={60:(intersection-1)}, name intersections={of=curve and line}] (60*\x+60:1) -- (intersection-1);}
\begin{scope}[xshift=3cm]
\draw (0,0) circle (1);
\foreach \x in {0,...,5} {
\draw (60*\x:1) -- ++(60*\x:-0.422) arc(60*\x+30:60*\x+60:1) -- (0,0);}
\end{scope}
\node at (1.5,0) {$\to$};
\end{tikzpicture}
\caption{$D_3^1\to$ RadGen}\label{fig:d31radgen}
\end{figure}

\noindent However, we may consider a union of two tiles that intersect at a vertex to be radially generated by two vertices. The only way for this to happen is that $\eta\cap\eta'=\alpha_p(\eta)$ contains a point other than $p$. Since this is exactly the criteria we considered when constructing the locus $\mathcal{R}_q$ in Theorem~\ref{thm:uncountable}, the exceptional tiling $D_3^1$ is the only tiling that may appear in this way.

\section{Enumerating members of $C_{n,k}^t$}\label{sect:enumc}

Suppose in a tiling from the family $C_{n,k}^t$ we have $k$ adjacent tiles, radially generated about a common point. Then their union is a symmetric $n$\--wedge, and we may locally `flip' all of these tiles in the lines of symmetry to produce a different monohedral tiling of the disk, using copies of the same fundamental tile.

Since it takes $2n$ copies of a symmetric $n$\--wedge to tile the disk, and since the construction involves subtiling each of these with $k$ tiles, $C_{n,k}^t$ contains $2nk$ tiles. Since each tile is radially generated, they may be arranged such that each tile touches the centre of the disk and the tiling has cyclic symmetry of order $2nk$. Since this may be done in two ways (up to reflection), we choose one of these and fix this as positive orientation, and colour the tiles `black'. Then we may flip up to $2n$ unions of $k$ adjacent tiles to get a new tiling. If the number we flip is less than $2n$ then we have some freedom of which we flip, as seen in Figure~\ref{fig:c320}. We say the flipped tiles have negative orientation, and colour the tiles `red'. It is our goal to count all different tilings that may be obtained in this way.

\begin{figure}[htbp]
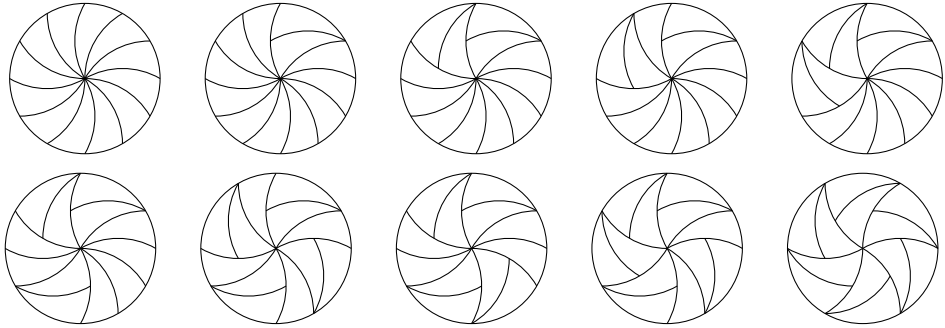
\centering
\C{1,2,3,4,5,6,7,8,9,10,11,12}{}{1}\quad
\C{1,3,4,5,6,7,8,9,10,11,12}{2}{1}\quad
\C{1,3,5,6,7,8,9,10,11,12}{2,4}{1}\quad
\C{1,3,4,6,7,8,9,10,11,12}{2,5}{1}\quad
\C{1,3,5,7,8,9,10,11,12}{2,4,6}{1}
\\\medskip
\C{1,3,5,6,7,9,10,11,12}{2,4,8}{1}\quad
\C{1,3,4,6,7,9,10,12}{2,5,8,11}{1}\quad
\C{1,3,5,6,7,9,11,12}{2,4,8,10}{1}\quad
\C{1,3,5,7,9,10,12}{2,4,6,8,11}{1}\quad
\C{2,4,6,8,10,12}{1,3,5,7,9,11}{1}
\caption{Some members of $C_{3,2}^0$}\label{fig:c320}
\end{figure}

\begin{definition}
The {\em necklace number} $\n_{a,b}$ counts the number of beaded necklaces that can made using $a$ black beads and $b$ red beads up to cyclic symmetry.
\end{definition}

\noindent Note that necklaces do not consider equivalence under reflective symmetry; such objects are called {\em bracelets}. This is the same as our equivalence of tilings, and so counting the number of tilings is equivalent to adding up appropriate necklace numbers.

For any tiling in $C_{n,k}^t$ we may count the number $ik$ of `black' tiles and the number $(2n-i)k$ of `red' tiles. The configuration of edges of the boundary of the disk corresponds exactly to a necklace; there are $i$ black edges and $(2n-i)k$ red edges. Hence we may compute the number of tilings in the family $C_{n,k}^t$ as
\[|C_{n,k}^t|=2\sum_{i=0}^{2n}\n_{i,(2n-i)k},\]
where the $2$ multiplying the sum accounts for the fact that we fixed the orientation, and any mirror image of a tiling is also a tiling.

In the exceptional case $k=1$ we have $|C_{n,1}^t|=2$ since any flips give the same tiling, so the only two family members are mirror images of each other. Table~\ref{tab:cnkt} gives some computed results.

\begin{table}[htbp]
    \centering
    \begin{tabular}{r|ccccc}
        $n=$ &3     &5     &7        &9          &11          \\\hline
    $k=1$    &2     &2     &2        &2          &2           \\
    2        &62    &1,532   &50,830    &1,855,110     &71,292,624    \\
    3        &116    &6402  &446,930   &34,121,322   &2,741,227,176  \\
    4        &200   &19,884  &2,460,462  &332,112,068  &47,162,138,964 \\
    5        &318   &51,128 &10,106,370  &2,177,193,500 &493,416,845,604
    \end{tabular}
    \caption{Some values of $|C_{n,k}^t|$}
    \label{tab:cnkt}
\end{table}

\noindent For fixed $a$ the coefficient of $x^b$ in the generating function
\[f_a(x)=\sum_{b=0}^\infty\n_{a,b}x^b=\frac{1}{a}\sum_{d|a}\frac{\varphi(d)}{(1-x^d)^{a/d}}\]
gives the necklace number $\n_{a,b}$ (see~\cite{polya}). Using the standard expansion
\[\frac{1}{(1-x^d)^{a/d}}=\sum_{r=0}^\infty\binom{r+a/d-1}{a/d-1}x^{dr},\]
we may write
\[\n_{a,b}=\frac{1}{a}\sum_{d:d|a,d|b}\varphi(d)\binom{b/d+a/d-1}{a/d-1}.\]

\begin{prop}
Combining the above arguments, we find that $|C_{n,1}^t|=2$, and for $k\ge2$ we have
\[|C_{n,k}^t|=2\sum_{i=0}^{2n}\sum_{\substack{d:d|i,\\d|(2n-i)k}}\frac{\varphi(d)}{i}\binom{(2n-i)k/d+i/d-1}{i/d-1}.\]
\end{prop}

\begin{definition}
Let $g$ be a function and $P$ be a predicate. We define the function
\[[g(k)]_{P(k)}=\left\{\begin{array}{cl}g(k),&P(k)\\0,&!P(k).\end{array}\right.\]
\end{definition}

\noindent For fixed values of $n$ this allows to simplify the expression for $|C_{n,k}^t|$ to a function which is the sum of a polynomial of degree $2(n-1)$, and a quasi\--periodic polynomial of lower degree. For example,
\[|C_{3,k}^t|=\frac{1}{60}k^4+\frac{5}{6}k^3+\frac{67}{12}k^2+\frac{61}{6}k+\frac{57}{5}+\left[1\right]_{2|k}+\left[\frac{8}{5}\right]_{5|k},\]

\[|C_{5,k}^t|=\frac{1}{181440}k^8+\frac{11}{1680}k^7+\frac{11527}{30240}k^6+\frac{973}{180}k^5+\frac{245269}{8640}k^3+\frac{3124847}{45360}k^2+\frac{10921}{315}k+\frac{1682}{126}+\]
\[+\left[\frac{1}{4}k+\frac{3}{2}\right]_{2|k}+\left[\frac{2}{81}k^2+\frac{10}{9}k+\frac{18}{9}\right]_{3|k}+\left[1\right]_{4|k}+\left[\frac{12}{7}\right]_{7|k}+\left[\frac{4}{3}\right]_{9|k}.\]

\section{Conclusions \& Conjectures}

We return to our list of questions from Section~1. Some of these now have definitive answers.

\begin{enumerate}
\item Tilings of type $D_n^t$ contain $4n$ tiles: $2n$ of these intersect the centre, and $2n$ of these do not. The family $\widetilde{C}_{n,k}^{t*}$ consists of tilings containing $2nk$ tiles, and this set of tiles may tile the disk in four different ways: two of these ways are radially generated tilings where every tile intersects the centre, and for the other two ways (obtained by `flipping' all $k$\--tuples of adjacent tiles forming a wedge) $2n$ tiles intersect the centre and $2n(k-1)$ do not. The situation for $C_{n,k}^t$ is a little more complicated as one may `flip' any number of adjacent $k$\--tuples of tiles forming a symmetric wedge, not necessarily all such tiles. So in this case the number of tiles intersecting the centre is reduced by $k-1$ for each flip, and the number of possible flips is $u=0,1,\ldots,2n$. Hence the number intersecting the centre is $2nk-u(k-1)$, and the number not intersecting the centre is $u(k-1)$.
\item Tilings of type $D_n^t$ for $n\ge5$ have exactly $2n$ tiles intersecting the boundary and $2n$ tiles not intersecting the boundary. The same is true for $D_3^t$ but only provided $t>0$ since each tile is $D_3^0$ intersects the boundary, either on an edge or a vertex. Similarly, since the subtiling comes from radially generated subtiling based at a vertex, the tilings $\widetilde{C}_{n,k}^{t*}$ and $\widetilde{C}_{n,k}^t$ have tiles that don't intersect the boundary only when $t>0$.
\item There are several members of $C_{n,k}^t$ such that at least one tile does not intersect the centre and the tiling has trivial cyclic symmetry, e.g. many of the images in Figure~\ref{fig:c320}.
\end{enumerate}

The remaining answers are addressed by the following conjectures, evidenced by our results.

\begin{conj}\label{conj:a}
Any monohedral tiling of the disk is a subtiling of a radially generated tiling.
\end{conj}

\begin{conj}\label{conj:b}
A subtiling of a radially generated tiling is one of the following:
\begin{itemize}
\item It is itself a radially generated tiling.
\item It may be obtained from a wedge tiling tiling by splitting a symmetric wedge in half, as in $D_n^t$, $t\in[0,1)$,
\item It may be obtained from a wedge tiling may be obtained by splitting the wedge into radially generated tiles, as in $\widetilde{C}_{n,k}^{t*}$ and $C_{n,k}^t$,
\item It is the tiling $D_3^1$.
\end{itemize}
\end{conj}

\noindent According to Corollaries~\ref{prop:no3wedge} and~\ref{cor:2}, we have classified all radially generated tilings and all wedge tilings. By Propositions~\ref{d:prop} and~\ref{c:prop}, we have classified all tilings that may be obtained from wedge tilings by subtiling the wedges radially or symmetrically. So $D_3^1$ is an exceptional case in our classification but, as explained in Section~\ref{sect:subtilings}, it is exceptional in the sense that a union of two tiles is a symmetric wedge, and this may never occur for any other tiling.

Since our construction completely classifies monohedral tilings by radially generated tiles, as well as their subtilings obtained in the manner of Conjecture~\ref{conj:b}, this along with Conjecture~\ref{conj:a} imply the final conjecture.

\begin{conj}
The full classification of monohedral disk tilings is present in this paper.
\end{conj}

\noindent Hence we claim that:
\begin{itemize}
\item for any monohedral tiling of the disk, the centre may only intersect a tile at a vertex. 
\item other than symmetric radially generated tilings, there is no monohedral tiling of the disk that either has a line or symmetry, or that contains an odd number of tiles.
\end{itemize}

\end{document}